\documentclass{article}
\usepackage{amsthm,amsmath,hyperref,bm,graphicx,multirow} 
\usepackage{fullpage,cleveref,amssymb}
\usepackage[font=footnotesize]{caption}
\usepackage{listings}

\newtheorem*{alb-conj}{Albertson's Conjecture}
\newtheorem{prop}{Proposition}
\newtheorem{lem}[prop]{Lemma}
\newtheorem{thm}[prop]{Theorem}
\newtheorem{thmA}{Theorem}
\newtheorem*{thmTwo}{Theorem 2}

\renewcommand\ge\geqslant
\renewcommand\geq\geqslant
\renewcommand\le\leqslant
\renewcommand\leq\leqslant
\newtheorem{lemA}[thmA]{Lemma}

\newcommand\Cr{\textrm{cr}}
\newcommand\floor[1]{\lfloor#1\rfloor}
\title{Progress on Albertson's Conjecture}
\author{Daniel W. Cranston\thanks{College of William \& Mary, Department of Mathematics; \texttt{dcransto@gmail.com}}}
\begin{document}
\maketitle
\begin{abstract}
Albertson conjectured that every graph with chromatic number $r$ has
crossing number at least the crossing number of the complete graph $K_r$.  This conjecture
was proved for $r\le 12$ by Albertson, Cranston, and Fox; for $r\le 16$ by Bar\'{a}t and T\'{o}th;
and for $r\le 18$ by Ackerman.  Here we verify it for $r\le 24$; we also greatly
restrict the possibilities for counterexamples when $r\in\{25,26\}$.
In addition, we strengthen earlier work bounding the order of a minimum counterexample for each choice of $r$:
we exclude the possibility that $|G|\ge 2.82r$ and exclude the possibility that $1.228r\le |G|\le 1.768r$.
Finally, as $r$ grows, we extend the lower end of this range of excluded orders for a minimum counterexample.
In particular: if $r\ge 125{,}000$, then we exclude the possibility that $1.10r\le |G|\le 1.768r$; 
and if $r\ge 825{,}000$, then we exclude the possibility that $1.05r\le |G|\le 1.768r$.
\end{abstract}

\section{Introduction}
\label{intro-sec}
The \emph{crossing number}, $\Cr(G)$, of a graph $G$ is the minimum number of edge crossings over all plane
embeddings of $G$.  The chromatic number, $\chi(G)$, of $G$ is the minimum number of colors needed so 
that the endpoints of each edge get distinct colors.  The 4 Color Theorem states: 
If $\Cr(G)=0$, then $\chi(G)\le 4$.  It is natural to ask how quickly $\chi(G)$ can grow with $\Cr(G)$.
By letting $G$ be the complete graph $K_r$, we get $\chi(G) = \Omega(\Cr(G)^{1/4})$.  And it is 
straightforward (see the first paragraph in the proof of \Cref{first-thm}) to prove a converse: 
$\chi(G)= O(\Cr(G)^{1/4})$.  Albertson believed something stronger.
\smallskip

\begin{alb-conj}[\cite{ACF}] If $\chi(G)\ge r$, then $\Cr(G)\ge \Cr(K_r)$.
\end{alb-conj}
\smallskip

For $r\le 4$, Albertson's Conjecture is trivial, and for $r=5$ it is equivalent to the 4 Color Theorem.  
It was proved for $r\le 12$ by Albertson, Cranston, and Fox~\cite{ACF}; for $r\le 16$ 
by Barat and Toth~\cite{BT}; and for $r\le 18$ by Ackerman~\cite{ackerman}.
Another line of research focuses on bounding, for each choice of $r$, the order of a minimum hypothetical counterexample $G$, with $\chi(G)=r$ and $\Cr(G)<\Cr(K_r)$.  It was shown that the order $|G|$ of such a $G$ must satisfy $|G|\le 4r$ \cite{ACF}; $|G|\le 3.57r$ \cite{BT}; and $|G|\le 3.03r$ \cite{ackerman}.  
Recently Fox, Pach, and Suk~\cite{FPS} showed that if $\chi(G)=r$ and $|G|\le r(1.64-o(1))$, then $G$ is
not a counterexample to Albertson's Conjecture.  That is, given $\epsilon > 0$, there exists $r_{\epsilon}$
such that if $\chi(G)=r$, with $r\ge r_{\epsilon}$ and $|G|\le r(1.64-\epsilon)$, then $\Cr(G)\ge \Cr(K_r)$.
(However, the proof assumes $r>2^{70}\approx 10^{21}$, and does not
directly say anything about smaller $r$.)

In this note we progress in both of these directions: proving Albertson's Conjecture for more small 
values of $r$, and further restricting the orders of minimal counterexample for larger values of $r$.  
To state our 2 main results, we need the following definition, studied intensely by Gallai~\cite{gallai2}.  
A graph $G$ is \emph{$r$-critical} if 
$\chi(G)=r$ and $\chi(H)<r$ for every proper subgraph $H$ of $G$.  Every $r$-chromatic graph contains an 
$r$-critical subgraph $G$; so it suffices to prove Albertson's Conjecture for all $r$-critical graphs.
Below are our $2$ main results.

\begin{thm}
	\label{thm1}
	If $G$ is $r$-critical, then $\Cr(G)\ge\Cr(K_r)$ for both (a) $1.212r\le |G|\le 1.768r$ and
	(b) $2.812r \le |G|$.
\end{thm}

\begin{thm}
	\label{thm2}
	Let $G$ be an $r$-critical graph.  If $r\le 24$, then $\Cr(G)\ge \Cr(K_r)$.
	And if $r\le 26$ and $\Cr(G) < \Cr(K_r)$, then $(r,|G|)\in\{(25,48),(26,50),(26,51)\}$.
\end{thm}

The notion of $r$-critical graphs was introduced in the 1950s by Dirac, and since 
that time has played a fundamental role in the study of chromatic graph theory.  In the present paper, we 
discuss only what we need of this theory 
for our proofs.  But for a more detailed history, we recommend~\cite{sasha-notes} and~\cite{brooks-book}.

The exact value of $\Cr(K_r)$ is known only when $r\le 12$.  However, well-known drawings of $K_r$ (with the
vertices spaced equally on 2 concentric circles) show
that $\Cr(K_r)\le \floor{r/2}\floor{(r-1)/2}\floor{(r-2)/2}\floor{(r-3)/2}/4$.  This upper bound is tight when $r\le 12$ and is conjectured to be tight for all $r$.  In fact, it is known to be tight~\cite{BLS} up to a 
factor of $0.985$ when $r$ is sufficiently large.

Most progress on Albertson's Conjecture takes the following approach. (1) Prove a lower bound on the number
$m$ of edges in an $n$-vertex $r$-critical graph $G$. (2) Prove a lower bound $f(m,n)$ on $\Cr(G)$ in terms
of $m$ and $n$. (3) Show that $\Cr(G)\ge f(m,n)\ge \floor{r/2}\floor{(r-1)/2}\floor{(r-2)/2}\floor{(r-3)/2}/4
\ge \Cr(K_r)$.  When $G$ is $r$-critical and $|G|\le 3r$, the best bounds on $m$ were determined long ago; 
see \Cref{Gallai-edge-lem} and \Cref{KS-edge-lem}, below.  So, much of the recent progress on Albertson's 
Conjecture is due to improved lower bounds on the above function $f(m,n)$.  And indeed, that is also the case 
in the present paper; see \Cref{best-crossing-thm}.
Using this approach we prove: \Cref{thm1}(b) in \Cref{big-r-sec}, \Cref{thm1}(a) in \Cref{middle-r-sec},
and \Cref{thm2} in \Cref{tiny-r-sec}.

We should also mention briefly a more recent approach.  A \emph{weak immersion} is essentially a subdivision
where paths between branch vertices can share internal vertices or even pass through other branch vertices
(we give a more formal definition near the start of \Cref{small-r-sec}).  Fox, Pach, and Suk~\cite{FPS}
showed that if $G$ is $r$-critical and $|G|\le r(1.64-o(1))$, then $G$ contains a weak immersion $G'$ of 
$K_r$.  They proved that $\Cr(G')\ge \Cr(K_r)-r^3/2$ and found at least $r^3/2$ other crossings in $G-E(G')$.
As we mentioned above, the work in \cite{FPS} does not apply when $r\le 10^{21}$, so in \Cref{small-r-sec}
we adapt it to yield helpful results in this range, too.  Our full result there, \Cref{thm3}, is a bit 
technical, but we get the 
following consequences.  Let $G$ be $r$-critical and a counterexample to Albertson's Conjecture.
If $r\ge 125{,}000$, then $|G|/r\notin[1.10,1.23]$; and if $r\ge 825{,}000$, then $|G|/r\notin[1.05,1.23]$.

\section{Handling \texorpdfstring{$\bm{r}$}{$r$}-critical graphs \texorpdfstring{$\bm{G}$}{G} with 
\texorpdfstring{$\bm{|G|\ge 2.82r}$}{G >= 2.82r}}
\label{big-r-sec}
Much work on crossing numbers draws heavily on versions of the so-called \emph{Crossing Lemma}; see 
\Cref{best-crossing-thm}(ii).
For its history, see~\cite[Chapter~45]{PFTB}.  The ratio $m^3/n^2$ is best possible, but the optimal 
constant is not yet known.
The sharpest currently known version was recently proved~\cite{BK} by Bungener and Kaufmann.  
\begin{thmA}[\cite{BK}]
\label{best-crossing-thm}
	Let $G$ be a graph with $n$ vertices and $m$ edges.  
	(i) If $n \ge 3$, then $\Cr(G)\ge 5m-\frac{203}9(n-2)$.
	(ii) As a result if $m \ge 6.95n$, then $\Cr(G)\ge m^3/(27.48n^2)$.
\end{thmA}

In~\cite{ACF} it was proved that if $G$ is an $r$-critical graph and $|G|\ge 4r$, then
$\Cr(G)\ge \Cr(K_r)$.  In~\cite{BT} and in~\cite{ackerman}, respectively, the same conclusion was 
proved under the weaker hypotheses $|G|\ge 3.57r$ and $|G|\ge 3.03r$.
Here we prove the following strengthening.

\begin{thm}
If $G$ is an $r$-critical graph with $r\ge 15$ and $|G|\ge 2.8118r$, then $\Cr(G)\ge \Cr(K_r)$.
	\label{first-thm}
\end{thm}
\begin{proof}
Let $n:=|G|$, let $m:=|E(G)|$, and let $\alpha:=n/r$.  To keep the proof self-contained, we first handle the case that $\alpha \ge 3.435$; that is, $\alpha\ge 27.48/8$.
	This case uses an argument from \cite{ACF}.  Recall that every $r$-critical graph 
	has minimum degree at least $r-1$.  Thus, $m\ge n(r-1)/2\ge 7n$.  So by \Cref{best-crossing-thm}(ii) we get:
	\begin{align*}	
		\Cr(G)\ge &~\frac{1}{27.48}\frac{m^3}{n^2} \ge \frac{1}{27.48}\left(\frac{(r-1)n}2\right)^3\frac1{n^2} \\
		 = &~\frac{(r-1)^3n}{8\times 27.48} = \frac{r(r-1)^3\alpha}{8\times 27.48} \ge \frac{r(r-1)^3}{64}\ge \Cr(K_r).
	\end{align*}	
	As promised in the introduction and noted previously in~\cite{ACF}, even without a lower bound on 
	$\alpha$, when $r\ge 15$ we get $\Cr(G)\ge (r-1)^4/2^8$, and hence $r\le 1+4\Cr(G)^{1/4} = O(\Cr(G)^{1/4})$.

	\smallskip
	Now we handle the case that $3.5 \ge \alpha \ge 2.8118$.
	Our proof follows the approach of \cite{ackerman}, but uses the stronger tool in 
	\Cref{best-crossing-thm}(i).  The general idea is to choose a $k$-vertex subgraph of $G$ 
	(uniformly at random) to which we apply \Cref{best-crossing-thm}(i).  More formally, 
	let $t:={n \choose k}$, and let $G_1,\ldots, G_t$
	denote the $t$ (inherited drawings of the) $k$-vertex subgraphs of $G$.  Note that each crossing of $G$ appears in 
    exactly ${n-4\choose k-4}$ of the $G_i$; similarly, each edge of $G$ appears in ${n-2 \choose k-2}$ of the $G_i$.
	For each $G_i$, we denote by $m_i$ its number of edges. So by \Cref{best-crossing-thm}(i) we have $\Cr(G_i)\ge 5m_i - \frac{203(k-2)}9$.  By summing over all $G_i$, we get:

\begin{align}
    \Cr(G) & \ge \frac{1}{{n-4\choose k-4}}\sum_{i=1}^t\Cr(G_i) \ge
            \frac{1}{{n-4\choose k-4}}\sum_{i=1}^t\left(5m_i - \frac{203(k-2)}9\right) \nonumber \\
           & =5m\frac{{n-2\choose k-2}}{{n-4\choose k-4}} - \frac{203(k-2){n\choose k}}{9{n-4\choose k-4}} \nonumber \\
	   & = 5m\frac{(n-2)(n-3)}{(k-2)(k-3)} - \frac{203n(n-1)(n-2)(n-3)}{9k(k-1)(k-3)}\label{smart-bound}\\
           & \ge \frac{5(r-1)n}2\frac{(n-2)(n-3)}{(k-2)(k-3)} - \frac{203n(n-1)(n-2)(n-3)}{9k(k-1)(k-3)}\nonumber\\
           & = \frac{n(n-2)(n-3)}{2(k-3)}\left(\frac{5(r-1)}{k-2}-\frac{406(n-1)}{9k(k-1)}\right)\nonumber\\
           & = \frac{\alpha^3r(r-\frac2\alpha)(r-\frac3\alpha)}{2(k-3)}\left(\frac{5(r-1)}{k-2}-\frac{406(n-1)}{9k(k-1)}\right)\nonumber\\
           & \ge \frac{(\alpha^3r((r-2)((r-3)+3)}{2(k-3)}\left(\frac{5(r-1)}{k-2}-\frac{406(r-1)(\alpha+\frac{\alpha-1}{r-1})}{9k(k-1)}\right)\nonumber\\
           & = \frac{(\alpha^3r(r-1)(r-2)(r-3)}{2(k-3)}\left(\frac{5}{k-2}-\frac{406\alpha}{9k(k-1)}\right)+h(\alpha,r,k),\nonumber
\end{align}
where we let
\begin{align*}
    h(\alpha,r,k)&:=\frac{\alpha^3r(r-1)(r-2)}{2(k-3)}\left(\frac{15}{k-2}-\frac{406}{9k(k-1)}\left(3\alpha+3\frac{\alpha-1}{r-1}+\frac{r-3}{r-1}(\alpha-1)\right)\right)\\
    &~\ge \frac{\alpha^3r(r-1)(r-2)}{2(k-3)}\left(\frac{15}{k-2}-\frac{406}{9k(k-1)}\left(3\alpha+\frac{\alpha-1}{6}+(\alpha-1)\right)\right).\\
\end{align*}
	The final inequality for $\Cr(G)$ holds because $\alpha\ge 2.8$, so $(r-2/\alpha)(r-3/\alpha)>(r-.8)(r-1.2) > r(r-2)$.

    Recall that $\Cr(K_r)\le \frac1{64}r(r-1)(r-2)(r-3)$.  So, to prove the proposition for all $r$-critical 
    graphs with $n=\alpha r$, for some specific value of $\alpha$, it suffices to find $k$ such that both 
    (a) $\frac{\alpha^3}{2(k-3)}\left(\frac{5}{k-2}-\frac{406\alpha}{9k(k-1)}\right) - \frac1{64}\ge 0$ and 
    (b) $h(\alpha,r,k)\ge 0$.  We will use $2$ different choices of $k$, depending on the value of $\alpha$.

    \begin{figure}[!h]
    \includegraphics[scale=.5] {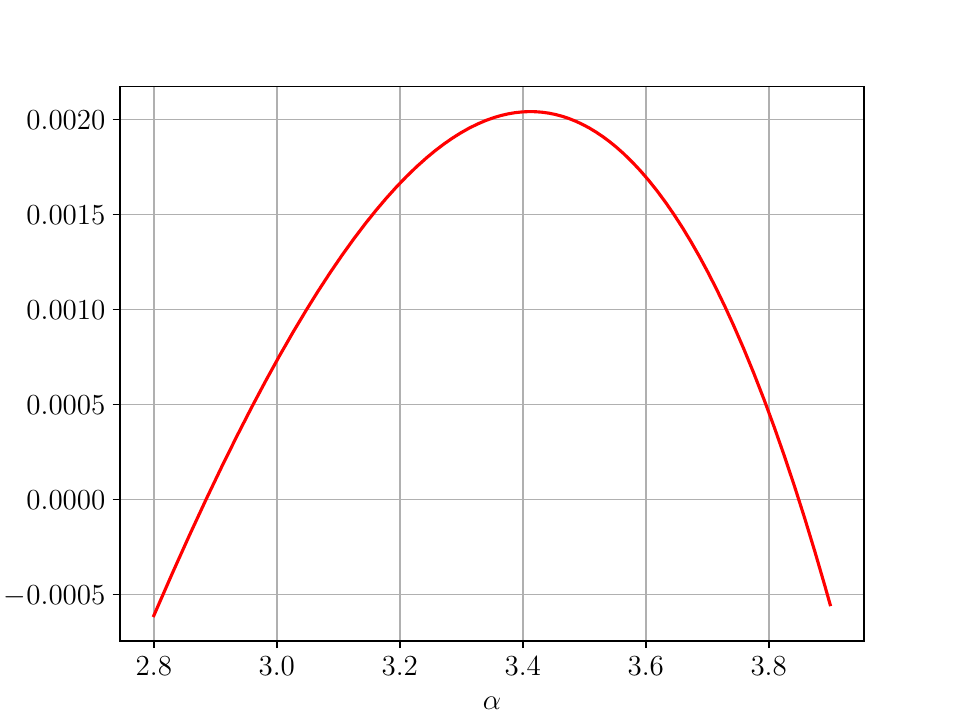}%
    \includegraphics[scale=.5] {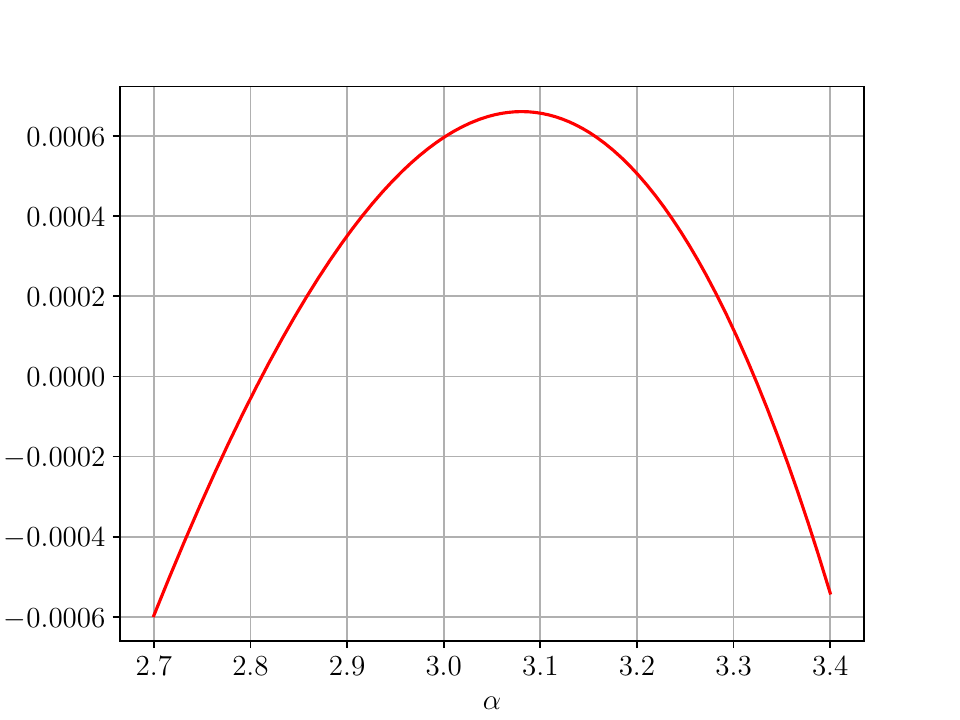}
			\captionsetup{width=.6\textwidth}
			\caption{Left: $\frac{\alpha^3}{2(40-3)}\left(\frac{5}{40-2}- 
			\frac{406\alpha}{9(40)(40-1)}\right) - \frac1{64}\ge 0$ when $\alpha\in(2.89,3.84)$.
			Right: $\frac{\alpha^3}{2(36-3)}\left(\frac{5}{36-2}- 
			\frac{406\alpha}{9(36)(36-1)}\right) - \frac1{64}\ge 0$ when $\alpha\in(2.8118,3.32)$.%
			\label{big-r-fig}}
    \end{figure}

    Let $k:=40$, so the left side of inequality (a) is 
	$\frac{\alpha^3}{2(40-3)}\left(\frac{5}{40-2}- \frac{406\alpha}{9(40)(40-1)}\right) - \frac1{64}$, 
	which is a downward opening quartic.  
    It is straightforward to check that it is positive on the interval $(2.89,3.84)$; 
	for example, see the left of \Cref{big-r-fig}, on the next page.
    Now we consider $h(\alpha,r,40)$.  We are only concerned with the function being nonnegative, 
    so we let $\hat{h}(\alpha,k):=h(\alpha,r,k)/(\alpha^3r(r-1)(r-2)/(2(k-3)))$ and show that
	$\hat{h}(\alpha,40)\ge 0$.  But $\hat{h}(\alpha,40)\ge15/38-406/(9\times40\times39))\times(19\alpha-17)/14$.
	So it is nonnegative when $\alpha\le 3.5$.

    Now let $k:=36$, so the left side of (a) is $\frac{\alpha^3}{2(36-3)}\left(\frac{5}{36-2}-
    \frac{406\alpha}{9(36)(36-1)}\right) - \frac1{64}$, which is again a downward opening quartic.  
    It is straightforward to check that it is positive on the interval $(2.8118,3.32)$; 
    for example, see the right of \Cref{big-r-fig}.
    Now we consider $h(\alpha,r,36)$.  Since we are only concerned with the function being nonnegative, 
    we let $\hat{h}(\alpha,k):=h(\alpha,r,k)/(\alpha^3r(r-1)(r-2)/(2(k-3)))$ and show that
	$\hat{h}(\alpha,36)\ge 0$.  But $\hat{h}(\alpha,36)\ge 15/34-406/(9\times36\times35))\times(59\alpha-17)/14$.
	So it is nonnegative when $\alpha\le 3.21$.
\end{proof}

\section{Handling \texorpdfstring{$\bm{r}$}{r}-critical graphs \texorpdfstring{$\bm{G}$}{G} 
with \texorpdfstring{$\bm{1.228r\le |G|\le 1.768r}$}{1.23r <= |G| <= 1.768r}}
\label{middle-r-sec}

Recall that when $G$ is $r$-critical, $\delta(G)\ge r-1$, so we trivially have $m\ge (r-1)|G|/2$.
But when also $r<|G|<2r$, we can greatly improve our lower bound on $m$.
\begin{lemA}[\cite{gallai2,KS}]
    \label{Gallai-edge-lem}
    If $G$ is an $n$-vertex $r$-critical graph, with $r\ge 4$ and $r+2\le n\le 2r-1$, then $|E(G)| 
    \ge ((r-1)n+(n-r)(2r-n)-2)/2$.
\end{lemA}

\Cref{Gallai-edge-lem} was proved by Gallai~\cite{gallai2} when $n\le 2r-2$ and extended to the case $n=2r-1$ 
by Kostochka and Stiebitz~\cite{KS}.
It is crucial in the proof of our next theorem, which is the main result of this section.
\begin{thm}
	\label{thm1b}
	If $G$ is $r$-critical and $1.228r\le |G|\le 1.768r$, then $\Cr(G)\ge \Cr(K_r)$.
\end{thm}
\begin{proof}
The full proof is a bit technical and tedious, but the main idea is simple.  We are looking for values of $\alpha$ such that there exists a value of $k$
	such that for all values of $r$ Inequality \eqref{ineq3.1} below holds.  

	\begin{align}
	\frac{5((r\!-\!1)\alpha r\!+\!(\alpha r\!-\!r)(2r\!-\!\alpha r)-2)}2\frac{(r\alpha\!-\!2)(r\alpha\!-\!3)}{(k\!-\!2)(k\!-\!3)} \!-\!
\frac{203r\alpha(r\alpha\!-\!1)(r\alpha\!-\!2)(r\alpha\!-\!3)}{9k(k-1)(k-3)}
		& \ge \frac{r(r\!-\!1)(r\!-\!2)(r\!-\!3)}{64} \label{ineq3.1}
\end{align}

	The right side of \eqref{ineq3.1} is an upper
	bound on $\Cr(K_r)$ and the left side is a lower bound on $\Cr(G)$ when $G$ is an $r$-critical graph
	with $|G|=\alpha r$; the latter arises from substituting the edge bound from \Cref{Gallai-edge-lem} 
	into Inequality~\eqref{smart-bound}.
	This $3$-dimensional parameter space (in $\alpha,r,k$) is harder to grasp
	intuitively, so we begin by projecting it to 2 dimensions.  When $r$ is large, the truth (or 
	falseness) of Inequality~\eqref{ineq3.1} essentially reduces to the truth (or falseness) of its
	restriction to the terms with a factor of $r^4$.  For this restriction, we can divide out the
	common factor of $r^4$, removing the dependence on $r$.  This yields Inequality \eqref{ineq3.2} below.

	\begin{align}
		\frac{5}2\frac{(\alpha+(\alpha-1)(2-\alpha))\alpha^2}{(k-2)(k-3)}-\frac{203}9\frac{\alpha^4}{k(k-1)(k-3)}-  \frac{1}{64} &\ge 0\label{ineq3.2}
\end{align}

	To better understand this constraint, we let $f(\alpha,k)$ denote the left side of Inequality 
	\eqref{ineq3.2}.  We consider \Cref{contour-fig}, which is a contour plot of $f(\alpha,k)$.  Now 
	Inequality \eqref{ineq3.2} holds on precisely the region 
	enclosed by the highest (innermost) level curve. %corresponding to the left side being nonnegative.

\begin{figure}[!h]
\begin{center}
\includegraphics[scale=.6]{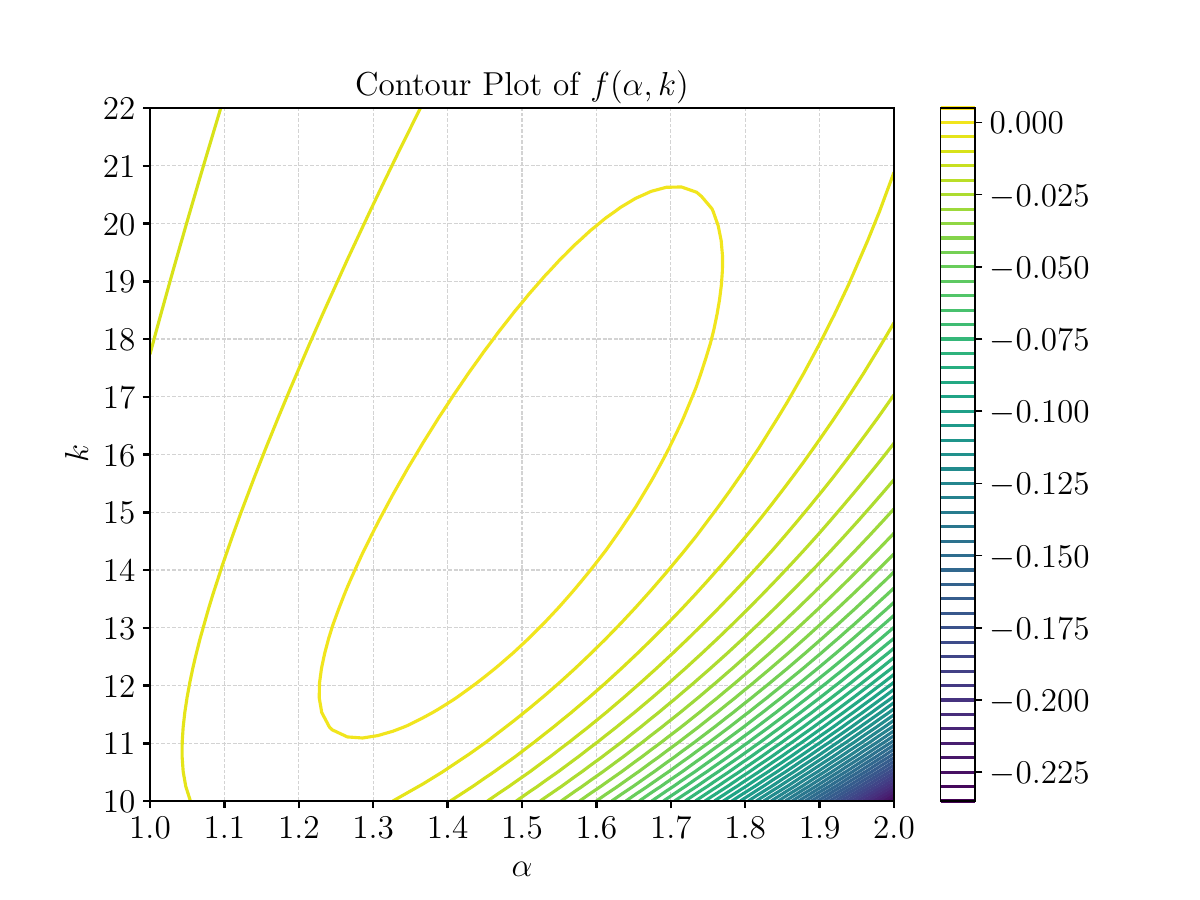}%
	\captionsetup{width=.54\textwidth}
	\caption{A contour plot of the left side of Inequality \eqref{ineq3.2}. The feasible region (where $f(\alpha,k)\ge0$) is the region inside of the innermost oval.\label{contour-fig}}
\end{center}
\end{figure}

	We see immediately that the smallest feasible value of $\alpha$ is about $\alpha=1.23$, which is
	valid only when $k=12$.  Similarly, the largest feasible value of $\alpha$ is about $\alpha=1.77$,
	valid only when $k=19$.  Furthermore, $k=12$ works for all $\alpha\in[1.23,1.43]$ and $k=19$
	works for all $\alpha\in[1.53,1.77]$.  To finish the proof, we also need a value of $k$ that
	works for all $\alpha\in[1.43,1.53]$.  Here we have flexibility.  From the contour plot, we see that
	we can choose any $k\in\{14,15,16,17\}$.  Somewhat arbitrarily, we choose $k=15$.

	We emphasize that the proof sketch above is not rigorous, because we have ignored lower order terms.
	However, it can be made rigorous, with essentially no change in the claimed bounds.  And that is
	what we do in the rest of the proof.  Below we provide the details.
	\bigskip

As noted above, for different parts of our range of $\alpha$ values, we will use different 
values of $k$ (19, 15, 12).  And for each value of $k$, we will 
get a polynomial $p_k(r,\alpha)$ in $r$ and $\alpha$, with each of $r$ and $\alpha$ raised to the power 
at most $4$.  We will need to determine the ranges of $\alpha$ such that $p_k(r,\alpha)\ge 0$ for all $r$.
There are various ways to perform this step, and we adopt the following.  We write $p_k(r,\alpha)$
as the sum of $4$ polynomials, and determine the appropriate range of $\alpha$
such that each of these polynomials is nonnegative (for all $r\ge 13$) and take the intersection of these 
ranges.  Our $i$th polynomial will consist primarily of the terms with $r^i$ as a (maximal) factor;
however, in a few terms we will vary from this approach.
We let
\begin{align*}
	p_{19}(r,\alpha) :=100\biggl[ &\frac{5((r-1)\alpha r+(\alpha r-r)(2r-\alpha r)-2)}2\frac{(r\alpha-2)(r\alpha-3)}{(19-2)(19-3)} -
\frac{203r\alpha(r\alpha-1)(r\alpha-2)(r\alpha-3)}{9(19)(19-1)(19-3)}\\
	&\left.-\frac{r(r-1)(r-2)(r-3)}{64}\right].
\end{align*}

Our goal is to show that $p_{19}(r,\alpha)\ge 0$ for all $r\ge 13$, over some range of values of $\alpha$.
The expression inside the brackets in the definition of $p_{19}(r,\alpha)$ is the left side of Inequality
\eqref{ineq3.1} minus the right, letting $k:=19$.
%\eqref{smart-bound}, by substituting the lower bound on $m$ from \Cref{Gallai-edge-lem}, 
%and subtracting an upper bound on $\Cr(K_r)$.
The multiplier $100$ is chosen, somewhat arbitrarily, to make more pleasant the coefficients in the 
expansion below.
Expanding and regrouping gives:

\begin{alignat*}{4}
	p_{19}(r,\alpha) =~ & r^4\left(-\frac{139325}{104652}\alpha^4 + \frac{125}{34}\alpha^3-\frac{125}{68}\alpha^2-\frac{25}{16}\right)
	&& +~r^3\left(\frac{214525}{34884}\alpha^3 - \frac{625}{34}\alpha^2+\frac{625}{68}\alpha+\frac{15}2\right)\\
	+~& r^2\left(-\frac{142675}{26163}\alpha^2 - \frac{375}{17}\alpha-\frac{7675}{272}+\frac{15}8r\right)
	&& +~r\left(-\frac{26525}{8721}\alpha+\frac{75}8\right)
\end{alignat*}
\smallskip

Let $p^{(4)}_{19}(\alpha):= -\frac{139325}{104652}\alpha^4 + \frac{125}{34}\alpha^3-\frac{125}{68}\alpha^2-\frac{25}{16}$;
let $p^{(3)}_{19}(\alpha):= \frac{214525}{34884}\alpha^3 - \frac{625}{34}\alpha^2+\frac{625}{68}\alpha+\frac{15}2$;
let $p^{(2)}_{19}(\alpha):= -\frac{142675}{26163}\alpha^2 - \frac{375}{17}\alpha-\frac{7675}{272}+\frac{15}8(13)$;
and let $p^{(1)}_{19}(\alpha):= -\frac{26525}{8721}\alpha+\frac{75}8$.
Note, when $r\ge 13$, that $p_{19}(r,\alpha) \ge 
r^4 \times p^{(4)}_{19}(\alpha)
+r^3 \times p^{(3)}_{19}(\alpha)
+r^2 \times p^{(2)}_{19}(\alpha)
+r \times p^{(1)}_{19}(\alpha)$.  (The reason we don't generally have equality is that when defining 
$p_{19}^{(2)}(\alpha)$ we replace $15r/8$ with the lower bound $15(13)/8$.)  So for any value of $\alpha$ 
such that $p_{19}^{(i)}(\alpha)\ge 0$ for all $i\in[4]$, we get that $p_{19}(r,\alpha)\ge 0$ whenever 
$r\ge 13$.  For the $4$ polynomials $p^{(i)}_{19}(\alpha)$ to be nonnegative, in order of decreasing $i$ we 
get (supersets of) the ranges: $\alpha\in (1.525, 1.7689)$ and $\alpha\in (1,\infty)$ and $\alpha\in (1,3.8)$ 
and $\alpha\in (1,3)$.  Thus, the intersection is the range $(1.525,1.7689)$; see \Cref{19-fig}.

\begin{figure}[!h]
\begin{center}
\includegraphics[scale=.635]{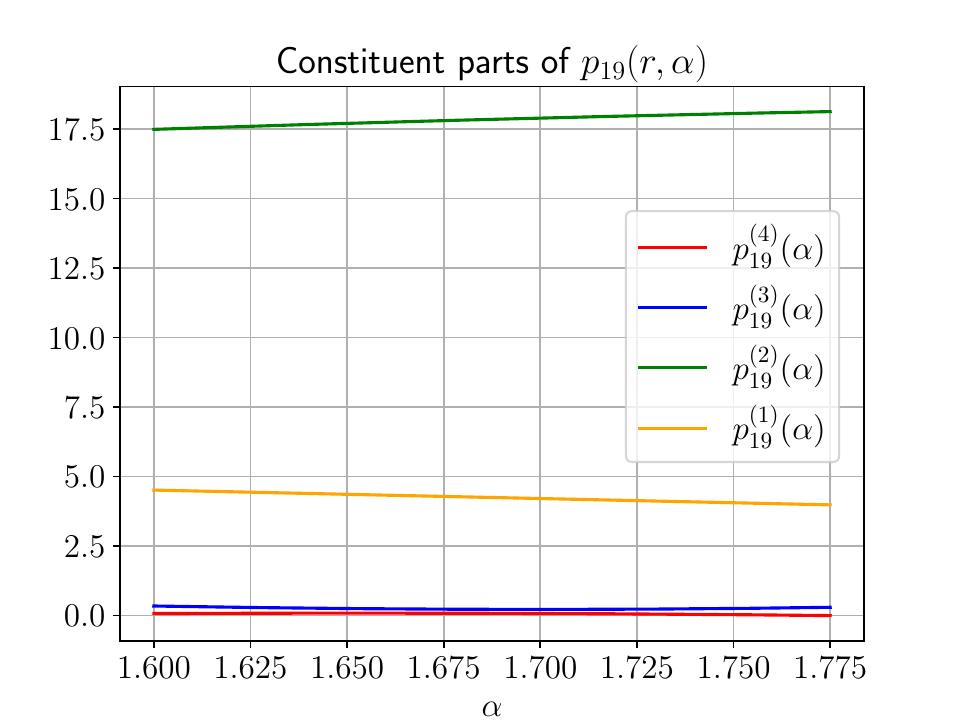}%
	\caption{The constituent parts of $p_{19}(r,\alpha)$ on $(1.600,1.775)$.\label{19-fig}}
\end{center}
\end{figure}

Next, we repeat the process above, letting $k:=15$, to handle a range when $\alpha$ is smaller.  We get
\begin{align*}
	p_{15}(r,\alpha) :=100\biggl[ &\frac{5((r-1)\alpha r+(\alpha r-r)(2r-\alpha r)-2)}2\frac{(r\alpha-2)(r\alpha-3)}{(15-2)(15-3)} -
\frac{203r\alpha(r\alpha-1)(r\alpha-2)(r\alpha-3)}{9(15)(15-1)(15-3)}\\
	&\left.-\frac{r(r-1)(r-2)(r-3)}{64}\right].
\end{align*}
After expanding and regrouping, we get:
\begin{alignat*}{4}
	p_{15}(r,\alpha) =~ & r^4\left(-\frac{2630}{1053}\alpha^4 + \frac{250}{39}\alpha^3-\frac{125}{39}\alpha^2-\frac{25}{16}\right)
	&& +~r^3\left(\frac{4135}{351}\alpha^3 - \frac{1250}{39}\alpha^2+\frac{625}{39}\alpha+\frac{70}8\right)\\
	+~& r^2\left(-\frac{12055}{1053}\alpha^2 + \frac{500}{13}\alpha-\frac{7575}{208}+\frac{5}8r\right)
	&& +~r\left(-\frac{1490}{351}\alpha+\frac{75}8\right).
\end{alignat*}
Let $p^{(4)}_{15}(\alpha):= -\frac{2630}{1053}\alpha^4 + \frac{250}{39}\alpha^3-\frac{125}{39}\alpha^2-\frac{25}{16}$;
let $p^{(3)}_{15}(\alpha):= \frac{4135}{351}\alpha^3 - \frac{1250}{39}\alpha^2+\frac{625}{39}\alpha+\frac{70}8$;
let $p^{(2)}_{15}(\alpha):= -\frac{12055}{1053}\alpha^2 - \frac{500}{13}\alpha-\frac{7575}{208}+\frac{5}8(13)$;
and let $p^{(1)}_{15}(\alpha):= -\frac{1490}{351}\alpha+\frac{75}8$.
Note, when $r\ge 13$, that $p_{15}(r,\alpha) \ge 
r^4 \times p^{(4)}_{15}(\alpha)
+r^3 \times p^{(3)}_{15}(\alpha)
+r^2 \times p^{(2)}_{15}(\alpha)
+r \times p^{(1)}_{15}(\alpha)$.  So for any value of $\alpha$ such that $p_{15}^{(i)}(\alpha)\ge 0$ for
all $i\in[4]$, we get that $p_{15}(r,\alpha)\ge 0$ whenever $r\ge 13$.
For the $4$ polynomials $p^{(i)}_{15}(\alpha)$, in order of decreasing $i$ we get (supersets of) the ranges: 
$\alpha\in (1.314, 1.648)$ and $\alpha\in (1,\infty)$ and $\alpha\in (1.1,2.2)$ and $\alpha\in (1,2)$.
Thus, the intersection is the range $(1.314,1.648)$; see \Cref{15-fig}.

\begin{figure}[!h]
\begin{center}
\includegraphics[scale=.635]{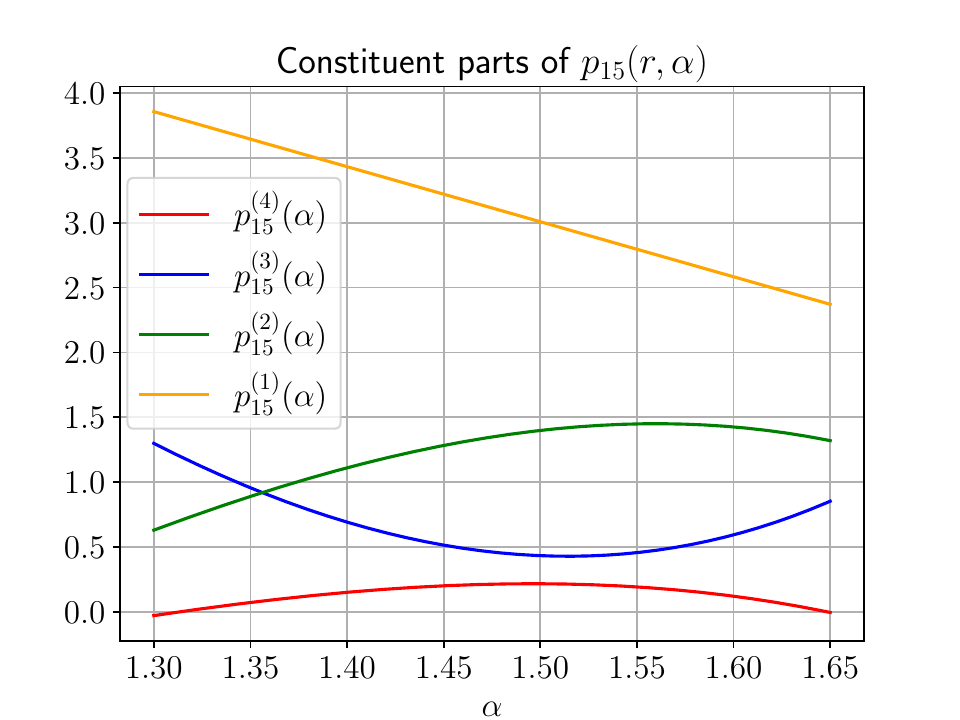}
		\caption{The constituent parts of 
	$p_{15}(r,\alpha)$ on $(1.30,1.65)$.\label{15-fig}} 
\end{center}
	\end{figure}

Finally, we repeat the process again, letting $k:=12$, to handle the range when $\alpha$ is smallest.  We get:
\begin{align*}
	p_{12}(r,\alpha) :=100\biggl[ &\frac{5((r-1)\alpha r+(\alpha r-r)(2r-\alpha r)-2)}2\frac{(r\alpha-2)(r\alpha-3)}{(12-2)(12-3)} -
\frac{203r\alpha(r\alpha-1)(r\alpha-2)(r\alpha-3)}{9(12)(12-1)(12-3)}\\
	&\left.-\frac{r(r-1)(r-2)(r-3)}{64}\right].
\end{align*}
After expanding and regrouping, we get:
\begin{alignat*}{4}
	p_{12}(r,\alpha) =~ & r^4\left(-\frac{12500}{2673}\alpha^4 + \frac{1000}{9}\alpha^3-\frac{500}{9}\alpha^2-\frac{25}{16}\right)
	&& +~r^3\left(\frac{20050}{891}\alpha^3 - \frac{500}{9}\alpha^2+\frac{250}{9}\alpha+\frac{69}8\right)\\
	+~& r^2\left(-\frac{5750}{243}\alpha^2 - \frac{200}{3}\alpha-\frac{2425}{48}+\frac{3}4r\right)
	&& +~r\left(-\frac{4700}{891}\alpha+\frac{75}8\right)
\end{alignat*}

\begin{figure}[!h]
\begin{center}
\includegraphics[scale=.625]{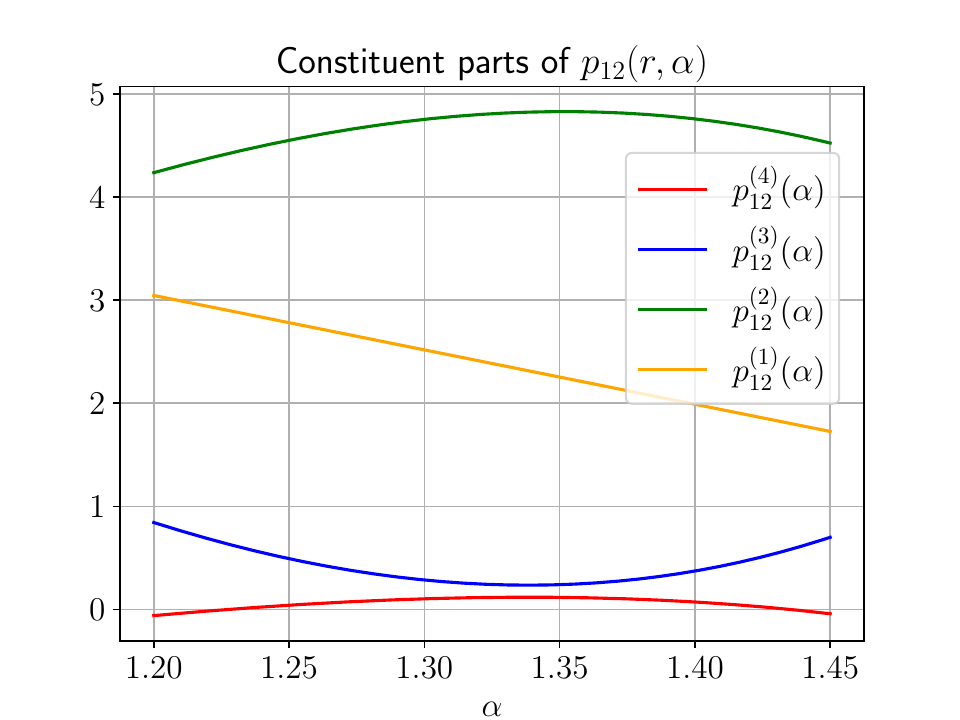}
		\caption{The constituent parts of $p_{12}(\alpha,r)$ on $(1.20,1.45)$.\label{12-fig}} 
\end{center}
	\end{figure}

Let $p^{(4)}_{12}(\alpha):= -\frac{12500}{2673}\alpha^4 + \frac{1000}{9}\alpha^3-\frac{500}{9}\alpha^2-\frac{25}{16}$;
let $p^{(3)}_{12}(\alpha):= \frac{20050}{891}\alpha^3 - \frac{500}{9}\alpha^2+\frac{250}{9}\alpha+\frac{69}8$;
let $p^{(2)}_{12}(\alpha):= -\frac{5750}{243}\alpha^2 + \frac{200}{3}\alpha-\frac{2425}{48}+\frac{3}4(13)$;
and let $p^{(1)}_{12}(\alpha):= -\frac{4700}{891}\alpha+\frac{75}8$.
Note, when $r\ge 13$, that $p_{12}(r,\alpha) \ge 
r^4 \times p^{(4)}_{12}(\alpha)
+r^3 \times p^{(3)}_{12}(\alpha)
+r^2 \times p^{(2)}_{12}(\alpha)
+r \times p^{(1)}_{12}(\alpha)$.  So for any value of $\alpha$ such that $p_{12}^{(i)}(\alpha)\ge 0$ for
all $i\in[4]$, we get that $p_{12}(r,\alpha)\ge 0$ whenever $r\ge 13$.
For the $4$ polynomials $p^{(i)}_{12}(\alpha)$, in order of decreasing $i$ we get (supersets of) the ranges: 
$\alpha\in (1.2274, 1.435)$ and $\alpha\in (1,\infty)$ and $\alpha\in (1,1.9)$ and $\alpha\in (1,1.7)$.
Thus, the intersection is the range $(1.2274,1.435)$; see \Cref{12-fig} on the next page.
\end{proof}

\section{Handling \texorpdfstring{\bm{$r\in\{19,20,21,22,23,24\}$}}{r=19,20,21,22,23,24} and most of \texorpdfstring{\bm{$r\in\{25,26\}$}}{r=25,26}}
\label{tiny-r-sec}

In this section, we consider small values of $r$.
For easy reference, we recall a few results of other authors.

\begin{lemA}[{\cite[Corollary~11]{BT}}]
    \label{BT-vert-cor}
    An $r$-critical graph with at most $r+4$ vertices contains a subdivision of $K_r$ (and thus satisfies Albertson's Conjecture).
\end{lemA}

\begin{lemA}[\cite{KS}]
    \label{KS-edge-lem}
    If $G$ is an $n$-vertex $r$-critical graph, with $n\ne 2r-1$, then $|E(G)|\ge n(r-1)/2+(r-3)$.
\end{lemA}

To handle the case in \Cref{KS-edge-lem} when $n=2r-1$, Bar\'{a}t and T\'{o}th used \Cref{KS-edge-lem} 
to prove the following.

\begin{lemA}[{\cite[Corollary 7]{BT}}]
    \label{BT-edge-cor}
    Let $G$ be an $n$-vertex $r$-critical graph with $r\ge 4$.  If $G$ does not contain a subdivision of $K_r$,
    then $|E(G)|\ge n(r-1)/2+(r-3)$.
\end{lemA}

Now we prove the main lemma of this section.

\begin{lem}
	Let $G$ be an $r$-critical graph with order $n$, with $r\ge 4$. 
	If $r<n<2r$, then inequalities \eqref{k=12-bound} and \eqref{k=22-bound} below hold.
	And if $G$ has no subdivision of $K_r$, then inequality \eqref{k=24-bound} below holds.
	\label{k-subgraph-bounds-lem}
\end{lem}
\begin{proof}
	For all 3 bounds, we fix a positive integer $k$ and sample uniformly the $k$-vertex subgraphs of 
	$G$, and apply \Cref{best-crossing-thm}(i) to our sampled graph.  Our proof follows that of the 
	second case in \Cref{first-thm} up to inequality~\eqref{smart-bound}, which we reproduce below 
	for easy reference.  

$$
	\Cr(G) \ge 5m\frac{(n-2)(n-3)}{(k-2)(k-3)}
	-\frac{203n(n-1)(n-2)(n-3)}{9k(k-1)(k-3)}.
$$

	For the first bound, let $k:=12$ and use $m\ge (n(r-1)+(n-r)(2r-n)-2)/2$ from \Cref{Gallai-edge-lem}.
	\begin{align}
		\Cr(G) \ge 5\frac{((r-1)n+(n-r)(2r-n)-2)}2\frac{(n-2)(n-3)}{(12-2)(12-3)}
		-\frac{203n(n-1)(n-2)(n-3)}{9(12)(12-1)(12-3)}. \label{k=12-bound}
	\end{align}

	For the second bound, let $k:=22$ and use $m\ge (n(r-1)+(n-r)(2r-n)-2)/2$ from \Cref{Gallai-edge-lem}.
	\begin{align}
		\Cr(G) \ge 5\frac{((r-1)n+(n-r)(2r-n)-2)}2\frac{(n-2)(n-3)}{(22-2)(22-3)}
		-\frac{203n(n-1)(n-2)(n-3)}{9(22)(22-1)(22-3)}. \label{k=22-bound}
	\end{align}

	For the third bound, let $k:=24$ and use the bound $m\ge n(r-1)/2+(r-3)$ from \Cref{BT-edge-cor}.

	\begin{align}
		\Cr(G) \ge 5\frac{((r-1)n+2(r-3))}2\frac{(n-2)(n-3)}{(24-2)(24-3)}
		-\frac{203n(n-1)(n-2)(n-3)}{9(24)(24-1)(24-3)}. \label{k=24-bound} 
	\end{align}

This concludes the proof.
\end{proof}
To prove the main result of this section, we need the following technical result.
The bounds it gives are generally weaker than those given by \Cref{k-subgraph-bounds-lem}, if we
instead choose appropriate values of $k$ in place of $24$.  But the calculations required to use it
are simpler, and we can thus exclude more values of $n$ all at once.

The proof closely mirrors that of analogous lemmas in~\cite{ackerman} and~\cite{BT}, so we just provide
a brief proof sketch.

\begin{lem}
    \label{cr-prop}
    Let $G$ be an $r$-critical graph with $r\ge 15$ and $n$ vertices and $m$ edges.
    Let $$\Cr(n,m,p) := \frac{5m}{p^2} - \frac{203n}{9p^3}+\frac{406}{9p^4}-0.5.$$ 
    Whenever $0.5\le p\le 1$, we have $\Cr(G)\ge \Cr(n,m,p)$.
\end{lem}
\begin{proof}[Proof sketch.]
	We select each vertex of $G$ with probability $p$ and apply \Cref{best-crossing-thm}(i), and linearity
	of expectation, to the resulting subgraph $\widehat{G}$.  We get the small error term $-0.5$ to 
	account for the (low probability) event that $\widehat{G}$ has at most $2$ vertices, in which case 
	\Cref{best-crossing-thm}(i) does not apply.
\end{proof}

Now we can prove \Cref{thm2}.  For convenience, we restate it below.

	\begin{table}[!b]
\begin{center}
	\renewcommand{\arraystretch}{1.15}
\begin{tabular}{c|c|c|c|c|c|c|cc|c}
	$r$ & $\Cr(K_r)\le$ & Lem.~\ref{BT-vert-cor} & Ineq.~\ref{k=12-bound} & Thm.~\ref{thm1b} & Ineq.~\ref{k=22-bound} & Ineq.~\ref{k=24-bound} & Lem.~\ref{cr-prop} & $p$ & possible $n$\\
	\hline
	\hline
	15 & ~~441 & [15, 19] & [17, 22] & [19, 26] & [22, 29] & [26, 36] & $[22,\infty)$ & 0.75 & -\\
	16 & ~~588 & [16, 20] & [19, 23] & [20, 28] & [24, 31] & [28, 38] & $[25,\infty)$ & 0.72 & -\\
	17 & ~~784 & [17, 21] & [20, 25] & [21, 29] & [25, 33] & [30, 40] & $[29,\infty)$ & 0.68 & -\\
	18 &  1008 & [18, 22] & [21, 26] & [23, 31] & [27, 35] & [32, 42] & $[32,\infty)$ & 0.65 & -\\
    \hline
	19 &  1296 & [19, 23] & [22, 28] & [24, 33] & [29, 36] & [35, 44] & $[35,\infty)$ & 0.62 & -\\
	20 &  1620 & [20, 24] & [23, 29] & [25, 35] & [30, 38] & [37, 46] & $[38,\infty)$ & 0.60 & -\\
	21 &  2025 & [21, 25] & [25, 31] & [26, 37] & [32, 40] & [39, 48] & $[42,\infty)$ & 0.58 & -\\
	22 &  2475 & [22, 26] & [26, 32] & [28, 38] & [34, 42] & [41, 50] & $[44,\infty)$ & 0.56 & -\\
	23 &  3025 & [23, 27] & [27, 34] & [29, 40] & [36, 44] & [44, 52] & $[48,\infty)$ & 0.54 & -\\
	24 &  3630 & [24, 28] & [28, 35] & [30, 42] & [37, 45] & [46, 54] & $[51,\infty)$ & 0.52 & -\\
	\hline
	\hline
	25 &  4356 & [25, 29] & [30, 36] & [31, 44] & [39, 47] & [49, 55] & $[54,\infty)$ & 0.50 & 48\\
	26 &  5148 & [26, 30] & [31, 38] & [32, 45] & [40, 49] & [52, 57] & $[58,\infty)$ & 0.50 & 50, 51\\
\end{tabular}
\end{center}
		\captionsetup{width=.515\textwidth}
		\caption{The proof of \Cref{thm2}. This table can be generated with a small bit of 
		Python code, which we provide in Appendix~\hyperref[appendix-B]{B}.\label{table1}}
	\end{table}

\begin{thmTwo}
	Let $G$ be an $r$-critical graph.  If $r\le 24$, then $\Cr(G)\ge \Cr(K_r)$.
	And if $r\le 26$ and $\Cr(G) < \Cr(K_r)$, then $(r,|G|)\in\{(25,48),(26,50),(26,51)\}$.
\end{thmTwo}
\begin{proof}
    Previous work~\cite{ACF, BT, ackerman} handles the case when $r\le 18$.  
	So now we can assume that $r\ge 19$.  We include lines in \Cref{table1} for $r\in\{15,16,17,18\}$ 
	just to highlight that our proof here also easily yields those results.

	For each $r$, we split the range of possible orders into 6 intervals.  The first of these is
	excluded by \Cref{BT-vert-cor}, which says that every $r$-critical graph $G$ on at most $r+4$ 
	vertices contains a subdivision of $K_r$ (and hence has $\Cr(G)\ge \Cr(K_r)$).  The second interval 
	is excluded by Inequality~\eqref{k=12-bound}, which samples 12-vertex subgraphs and applies 
	\Cref{best-crossing-thm}(i).  The third interval is excluded by \Cref{thm1b}, which says that no 
	counterexample $G$ to Albertson's conjecture has $1.228r\le |G|\le 1.768r$.

	The fourth and fifth intervals are excluded by Inequalities~\eqref{k=22-bound} 
	and~\eqref{k=24-bound}, which sample
	the $k$-vertex subgraphs of $G$, with $k=22$ and with $k=24$.  Finally, the sixth interval is 
	excluded by \Cref{cr-prop}, which takes a random subgraph, including each vertex with probability 
	$p$.  We also show our choice of $p$.  For each of these intervals, the lemma yields a lower 
	bound on $\Cr(G)$; we must verify, for each $n$ in the interval, that this lower bound is at least 
	the value $\floor{r/2}\floor{(r-1)/2}\floor{(r-2)/2}\floor{(r-3)/2}/4$ shown in the second column.

	The entire table can be generated with a short Python program, which we provide in 
	Appendix~\hyperref[appendix-B]{B}.
	%Appendix~\ref{appendix-B}.
\end{proof}

\section{Excluding more values of \texorpdfstring{$\bm{|G|}$}{|G|} when \texorpdfstring{$\bm{r}$}{r} is Large}
\label{small-r-sec}

Fox, Pach, and Suk~\cite{FPS} recently proved the following result.

\begin{thmA}[\cite{FPS}]
	\label{FPS-thm}
	For every $\epsilon>0$, there exists $r_{\epsilon}$ such that if $r\ge r_{\epsilon}$ and $\chi(G)=r$
	and $|G|\le r(1.64-o(1))$, then $G$ is not a counterexample to Albertson's Conjecture.
\end{thmA}
Asymptotically, \Cref{FPS-thm} is very interesting.
But one weakness of this result is that the proof assumes that $r>2^{70}\approx 10^{21}$, and it does not
directly say anything when $r$ is smaller.
When $|G|\ge 1.228r$, \Cref{thm1b} improves \Cref{FPS-thm}, since (a) we 
require no lower bound on $r$ and (b) we weaken the upper bound in the hypothesis to $|G|\le 1.768r$.  
In this section,
we adapt the argument of \cite{FPS} to also exclude many possible orders of counterexamples when 
$r\ll 10^{21}$.
Because we are focused only on the lower end of this range, we avoid many of the complications in~\cite{FPS},
which makes our proof shorter and simpler.

Gallai~\cite{gallai2} proved the following result, which is crucial for describing $r$-critical graphs $G$
with $|G|\le 2r-2$.  For example, it plays essential roles in the proof of \Cref{Gallai-edge-lem} and
in the proof of \Cref{FPS-immersion-lem}, below.
\begin{lemA}[\cite{gallai2}]
	\label{gallai-join-lem}
	Fix positive integers $r$ and $n$ with $n\le 2r-2$.  If $G$ is an $r$-critical graph on $n$
	vertices, then there exists a partition
	$$V(G) = V_1\uplus V_2\uplus \cdots \uplus V_t,$$
	where $t\ge 2$, such that $V_i$ is complete to $V_j$ whenever $i\ne j$, and the induced subgraph
	$G[V_i]$ is $k_i$-critical with $|V_i| \ge 2k_i-1$ whenenver $i\in [t]$.
\end{lemA}

A \emph{weak immersion} of a graph $H$ in a
graph $G$ is an injective map $\varphi$ of vertices of $H$ to vertices of $G$, and of the edges $vw\in V(H)$ 
to edge-disjoint $\varphi(v),\varphi(w)$-paths in $G$.  Each vertex in $G$ that is the image of a vertex in 
$H$ is a \emph{branch vertex}.
So every subdivision is a weak immersion, but not vice versa, since
an internal vertex of some path in $G$ might be shared among multiple paths, or could even be a branch vertex.

To prove the main result of this section, we need 2 more lemmas from~\cite{FPS}.
\begin{lemA}[\cite{FPS}]
	\label{FPS-immersion-lem}
Let $G$ be an $r$-critical graph with $|G|<1.4r-0.6$.  
	Fix a partition $\cup_{i=1}^tV_i$ of $V(G)$ as in \Cref{gallai-join-lem}.
Fix an arbitrary partition $U_i\uplus W_i$ of each 
$V_i$, with $|U_i|=r_i$, and let $W:=\cup_{i=1}^tW_i$.  
Now $G$ contains a weak immersion $G'$ of $K_r$ that does not use any of the edges in $G[W]$.  
\end{lemA}

\begin{lemA}[\cite{FPS}]
	\label{immersion-cr-lem}
	Fix a positive integer $r$, let $G'$ be a graph, and let $n:=|G'|$.  
	If $G'$ is a weak immersion of $K_r$, then $\Cr(G')\ge \Cr(K_r)-n(n-r)(n+2r)/8$.
\end{lemA}

\Cref{FPS-immersion-lem} is essentially Theorem~1.2(i) in~\cite{FPS}.
And \Cref{immersion-cr-lem} is not stated explicitly in~\cite{FPS}, but its proof is essentially the proof of 
Lemma~4.3~in that paper.  (For completeness, we include a proof of \Cref{immersion-cr-lem} in 
Appendix~\hyperref[appendix-A]{A}.)
By adapting arguments in~\cite{FPS}, we get that $\Cr(G')\ge \Cr(K_r)-n(n-r)(n+2r)/8$, when $G'$ is a weak 
immersion of $K_r$.  We then use \Cref{best-crossing-thm}(ii) to show that
$\Cr(G[W])\ge \Cr(K_r)-\Cr(G')$; thus, we have $\Cr(G)\ge \Cr(G')+\Cr(G[W]) \ge \Cr(K_r)$.
%upper bound on $\Cr(K_r)-\Cr(G')$.)
We can now prove the main result of this section.

\begin{thm}
	\label{thm3}
	Fix a positive integer $r$, and let $G$ be an $r$-critical graph with $n$ vertices where $n\le 1.23r$.
	If
	\begin{align}	
		27.48 &\le \frac{(r-1)^3(n-r-1)^3}{(n-1)^3n(n+2r)}, \label{finalish-ineq}
	\end{align}
	then $G$ is a not a counterexample to Albertson's Conjecture.
	In particular, this is true if $r\ge 125{,}000$ and $1.10r\le n \le 1.23r$; and it is true if $r\ge 825{,}000$ and $1.05r\le n\le 1.23r$.
\end{thm}
\begin{proof}
	To begin we show that the second and third statements follow from the first.  To simplify our calculations, we want to drop each additive $-1$ term, particularly in the numerator.  For both the second and third statement, $\min\{r,n-r\}\ge 12{,}500$.  So it suffices to show that $27.48\times (1-1/12{,}500)^{-6}\le 27.50\le r^3(n-r)^3/(n^4(n+2r))$.  Letting $\alpha:=n/r$, we rewrite this final expression as
	$$
	\frac{r^3(\alpha-1)^3r^3}{\alpha^4r^4(\alpha+2)r} = \frac{(\alpha-1)^3}{\alpha^4(\alpha+2)}r.
	$$
	It is easy to check that over any interval $(a,1.23]$, with $a>1$, this function of $\alpha$ is 
	minimized at $a$.  So, for the second statement, we need only verify that 
	$0.1^3\times 125{,}000/(1.1^4\times 3.1) > 27.5$; and this is true.  
	For the third statement, we check that
	$0.05^3\times 825{,}000/(1.05^4\times 3.05) > 27.5$; and again this is true.  

Now we prove the first statement.  Since $G$ is $r$-critical, $\delta(G)\ge r-1$, so $m\ge n(r-1)/2$.
	Consider the partition $V(G)=\cup_{i=1}^tV_i$ guaranteed by \Cref{gallai-join-lem}, and a
	partition $V_i=U_i\uplus W_i$ as in \Cref{FPS-immersion-lem}.
	Since $|W|=n-r$, the average value of $|E(G[W])|$, over all partitions, is at least 
	$\frac{n(r-1)}2\frac{{{n-r}\choose 2}}{{n\choose 2}} = \frac{(r-1)(n-r)(n-r-1)}{2(n-1)}$.  
	So some parition exists with $W$ inducing at least this many edges; fix such a partition.  
	By \Cref{FPS-immersion-lem}, we know that $G$ contains a weak immersion $G'$ of $K_r$ that does not 
	use any edge of $G[W]$.  And by \Cref{immersion-cr-lem}, we have $\Cr(G')\ge 
	\Cr(K_r)-n(n-r)(n+2r)/8$.  So it suffices to show that $\Cr(G[W])\ge n(n-r)(n+2r)/8$.  For this 
	we use \Cref{best-crossing-thm}(ii).  
	So we need $|E(G[W])|^3/(27.48|W|^2)\ge n(n-r)(n+2r)/8$.
	Recall that $|W|=n-r$.  Now substituting and simplifying gives \eqref{finalish-ineq}, as desired.
	Finally, we much check that \Cref{best-crossing-thm}(ii) does in fact apply. That is, we must verify 
	that 
	\begin{align}
		\label{final-ineq}
		\frac{(r-1)(n-r-1)}{2(n-1)}\ge 6.95.
	\end{align}
	Note that $\max\{r-1,n-r-1\}=r-1\le n-1=\min\{n-1,n,n+2r\}$.  Thus, the hypothesis implies that
	$n-r-1\ge 27.48$.  As a result, we have $(n-1)/(r-1) \le 1.5$.  So \Cref{final-ineq} is implied
	by $n-r-1\ge 21 \ge 2\times1.5\times 6.95$.  And this is true, as noted just above.
\end{proof}
\vspace{-0.1in}%

\phantomsection%
\label{appendix-A}%
\section*{Appendix A: Proof of \texorpdfstring{\Cref{immersion-cr-lem}}{Lemma~I}}

The following lemma and proof are slight variations on ones in~\cite{FPS}; we only include them
here for completeness.

\addtocounter{thmA}{-1}
\begin{lemA}[\cite{FPS}]
	%\label{immersion-cr-lem}
	Fix a positive integer $r$, let $G'$ be a graph, and let $n:=|G'|$.  
	If $G'$ is a weak immersion of $K_r$, then $\Cr(G')\ge \Cr(K_r)-n(n-r)(n+2r)/8$.
\end{lemA}

The basic idea of the proof is to take a plane embedding of a weak immersion of $K_r$, 
modify it to get a plane embedding of a subdivision of $K_r$, and then prove an upper bound 
(which is $n(n-r)(n+2r)/8$) on the number of crossings that were created during this modification process.

\begin{proof}
	Consider a graph $G'$ that is a weak immersion of $K_r$; let $v_1,\ldots, v_r$ be the
	branch vertices, and let $P_{ij}$ be the path in $G'$ connecting $v_i$ and $v_j$.
	Fix a drawing $\tilde{G}$ of $G'$ in the plane with $\Cr(G')$ crossings.  We will transform 
	$\tilde{G}$ into a plane embedding $\widehat{G}$ that contains a subdivision of $K_r$.

	Find a place where a path $P_{ij}$ passes through a vertex $v_h$ and redraw a small neighborhood
	around $v_h$ so that instead of $P_{ij}$ passing through $v_h$, the path $P_{ij}$ goes around
	$v_h$, either clockwise or counterclockwise, whichever creates fewer new crossings.  We repeat
	this step of finding and rerouting a path $P_{ij}$ that passes through a vertex $v_h$ until we 
	arrive at a plane embedding of a graph $\widehat{G}$ that contains a subdivision of $K_r$.

	Now we count crossings in $\widehat{G}$.  We group them into two types:
	type 1 crossings, which already appeared in $\tilde{G}$; and type 2 crossings, which were
	created in a redrawing steps.  We will show that $\widehat{G}$ has at most $n(n-r)(n+2r)/8$
	type 2 crossings.
	We further split the type 2 crossings into ones occurring in a small neighborhood around a branch
	vertex, and ones occurring in a small neighborhood around a~non-branch~vertex.

	Since each branch vertex $v_h$ is the endpoint of $r-1$ paths to other branch vertices, the number
	of times that $v_h$ appears as an internal vertex of some path $P_{ij}$ is 
	at most $(d(v_h)-(r-1))/2\le (n-r)/2$.  Each time that such a path $P_{ij}$ is rerouted from 
	passing through $v_h$, this rerouting creates at most $(d(v_h)-2)/2<n/2$ crossings.  So the total
	number of type 2 crossings in $\widehat{G}$ that are close to branch vertices is at most $r(n-r)n/4$.

	Now consider the type 2 crossings near non-branch vertices.  At each of the $n-r$ non-branch vertices
	$v_h$, we have at most $(n-1)/2$ paths $P_{ij}$ passing through $v_h$.  And when each of these paths
	$P_{ij}$ is rerouted around $v_h$ (either clockwise or counterclockwise), in the worst case each
	pair of these paths creates at most $1$ new crossing.  So the total number of these crossings is at 
	most $(n-r){(n-1)/2\choose 2}\le (n-r)((n-1)/2)^2/2\le (n-r)n^2/8$.  Thus, the total number 
	of (type 2) crossings that we create when transforming $\tilde{G}$ to $\widehat{G}$ is at most
	$r(n-r)n/4+(n-r)n^2/8 = n(n-r)(n+2r)/8$, as claimed.
\end{proof}
\vspace{-.22in}~

\phantomsection%
\label{appendix-B}%
\section*{Appendix B: Python Code to Generate \texorpdfstring{\Cref{table1}}{Table 1}}
\vspace{-.215in}~

\lstset{language=python, basicstyle=\small}
\begin{lstlisting} 
from sympy import *; n, r= symbols('n r')
p_array = [.75, .72, .68, .65, .62, .60, .58, .56, .54, .52, .50, .50]

def round_interval(lower_upper): 
    # round lower and upper endpoints up and down (unless upper is oo, i.e., infinity) 
    # lower_upper always string of form '(a <= n) & (n <= b)' or '(a <= n) & (n < oo)'
    lower = "["+str(ceiling(float(lower_upper.split("&")[0].split("<")[0].strip("( "))))
    upper_tmp = lower_upper.split("&")[1].split("<")[1].strip("=) ")
    if (upper_tmp == "oo"): upper = upper_tmp+")" 
    else: upper = str(floor(float(upper_tmp)))+"]"
    return lower+","+upper;

print("n    K_r  Lem C  Ineq 2   Thm 4  Ineq 3  Ineq 4   Lem 6  prob");

for r in range(15,27):         
    # target below is a known upper bound on cr(K_r)  
    target = floor(r/2)*floor((r-1)/2)*floor((r-2)/2)*floor((r-3)/2)/4
    targ4  = str(target).rjust(4); # add leading space if target is only 3 digits
    lemC   = "["+str(r)+","+str(r+4)+"]";
    thm4   = "["+str(ceiling(1.228*r))+","+str(floor(1.768*r))+"]"

    # ineq2 uses k=12, with Gallai's edge bound
    ineq2  = solve(5.0*((r-1)*n+(n-r)*(2*r-n)-2)/2*(n-2)*(n-3)/((12-2)*(12-3))
             -203*n*(n-1)*(n-2)*(n-3)/(9*12*(12-1)*(12-3))- target>=0,n)

    # ineq3 repeats ineq2, but with k=22, (still with Gallai's edge bound)
    ineq3  = solve(5.0*((r-1)*n+(n-r)*(2*r-n)-2)/2*(n-2)*(n-3)/((22-2)*(22-3))
             -203*n*(n-1)*(n-2)*(n-3)/(9*22*(22-1)*(22-3))- target>=0,n)

    # ineq4 repeats ineq2, but with k=24, now with Kostochka-Stiebitz edge bound
    ineq4  = solve(5.0*((r-1)*n+2*(r-3))/2*(n-2)*(n-3)/((24-2)*(24-3))
             -203*n*(n-1)*(n-2)*(n-3)/(9*24*(24-1)*(24-3))- target>=0,n)

    p      = p_array[r-15]  # probability with which we include each vertex
    lem6   = solve(5.0*((r-1)/2*n+r-3)/p**2   # again uses K-S edge bound
             -203*n/(9*p**3)+406/(9*p**4)-0.5-target>=0,n)

    print(r, ": ", targ4, " ", lemC, " ",  round_interval(str(ineq2)), 
          " ", thm4, " ", round_interval(str(ineq3)), " ", 
	  round_interval(str(ineq4)), " ", round_interval(str(lem6)), 
	  " ","{:04.2f}".format(p), sep="")
\end{lstlisting}
   % # Clean print out
%\fi

\section*{Acknowledgments}
Thanks to the American Institute of Mathematics for its hospitality during a weeklong workshop in October 
2024 entitled ``Albertson conjecture and related problems'' where I began work on this project. Thanks 
also to a working group during that workshop for some early discussions
on this topic.  Special thanks to G\'{e}za T\'{o}th for helpful feedback on an earlier draft of this paper.
Incorporating his suggestions greatly improved the clarity of the proof of \Cref{thm1b}.
Finally, thanks to Beth Cranston for help with aspects of Python.

\bibliographystyle{habbrv}
{{\bibliography{albertson}}}
%{\footnotesize{\bibliography{albertson}}}
\end{document}